\documentclass[11pt]{amsart}
\usepackage{amsmath} 
\usepackage{amsthm}
\usepackage{amssymb} 
\usepackage{amsfonts}
\usepackage{graphicx} 
\usepackage{verbatim}  

\newtheorem{theorem}{Theorem}[section]
\newtheorem{proposition}[theorem]{Proposition}
\newtheorem{lemma}[theorem]{Lemma}

\newtheorem{observation}[theorem]{Observation}
\newtheorem{claim}[theorem]{Claim}
\newtheorem{exercise}[theorem]{Exercise}

\newtheorem*{technical}{Technical Clarification}
\newtheorem*{aside}{Aside}
\newtheorem*{notation}{Notation Convention}

\newcommand{\bbZ}{\mathbb{Z}}
\newcommand{\bbR}{\mathbb{R}}
\newcommand{\pluralspace}{\nolinebreak\hspace{1pt}\nolinebreak } 
\newcommand{\bdy}{\partial}
\newcommand{\rel}{\textrm{ rel }}
\newcommand{\scaps}{\,\cap\,}
\newcommand{\scups}{\,\cup\,}
\newcommand{\interior}{\mathring}
\newcommand{\abs}[1]{|#1|}
\newcommand{\circleplus}{\oplus}
\newcommand{\isom}{\cong}
\newcommand{\point}{\mathrm{point}}
\newcommand{\setdifference}{\mathbin{ {\raise.1ex\hbox{$-$}} {\mkern-3mu}
  {\raise-.2ex\hbox{$\scriptscriptstyle \setminus$}}}} 
\newcommand{\Embeds}{\textsl{Embeds}}
\newcommand{\Diffeo}{\mathrm{Diffeo}}
\newcommand{\mathendash}{\text{--}}

\DeclareMathOperator{\connectsum}{\#}
\DeclareMathOperator{\proj}{proj}
\DeclareMathOperator{\nbd}{nbd}

\DeclareMathOperator{\interioroperator}{interior}
\DeclareMathOperator{\frontier}{frontier}

\usepackage[margin=.5in]{geometry}
      
\title[The $4$-Dimensional Light Bulb Theorem (after David Gabai)]{The $4$-Dimensional Light Bulb Theorem (after David Gabai)}
\author[Robert D. Edwards]{Robert D. Edwards}

\begin{document}

\thanks{Acknowledgement: My heartiest thanks to Maggie Miller of Princeton University for her invaluable help in turning my crude ascii manuscript and hand drawings into this electronic document. Without her help it would not have happened.}

\begin{abstract}
In this note I present my understanding of, that is to say the way I look at, David Gabai's proof of his recent 4-Dimensional Light Bulb Theorem (4D-LBT). His construction, entirely smooth, is an ingenious amalgam of classical moves, and represents the first new hands-on advance in constructive smooth 4-manifold theory, that I am aware of, in a long time.
\end{abstract}

\maketitle

\setcounter{tocdepth}{2}
\setcounter{equation}{0}
\section{Introduction}

David Gabai recently posted \cite{gabai} a proof of what he fittingly calls the 4-Dimensional Light Bulb Theorem (4D-LBT, for short).

\begin{theorem}[4D-LBT]
Everything is smooth. Suppose $\Sigma \subset S^2 \times S^2$ is an embedded 2-sphere which intersects some $S^2 \times \{\point\}$ in a single point, transversely. Then $\Sigma$ is unknotted. That is, $\Sigma$ can be diffeotoped to one of $\bbZ$-many canonical positions, indexed by the degree of the projection of $\Sigma$ to the first factor $S^2$ (after orienting the various $2$-spheres here).
\end{theorem}

Thus, for some $n \in \bbZ$, $\Sigma$ is isotopic to $\{z_0\} \times S^2 \connectsum_{1 \le i \le \abs{n}} S^2 \times \{z_i\}$, where the $z_i$\pluralspace s, are distinct and the connect-summing is done at the intersection points in a consistent manner to make the degree of the projection correct. 

As its name suggests, this theorem may be regarded as a generalization of the classical 3-D Light Bulb Theorem, where $S^2 \times S^2$ is replaced by $S^2 \times S^1$ and $\Sigma$ is a $1$-sphere, and the conclusion is that $\Sigma$ can be diffeotoped to $\{\point\} \times S^1$.  Also this Theorem may be compared directly with the analogous result in dimension 2, where $S^2  \times S^2$ is replaced by $S^1 \times S^1$ and $\Sigma$ is a 1-sphere, that is $S^1$. We note that Gabai's Theorem completes the entire spectrum of such theorems, for $n = 1,2,3, ...$.

\begin{technical}\rm We note that the isotopy class of the {\it{embedding}} whose image is $\Sigma$, depends as well on the degree of its projection to the second factor, which can be $\pm 1$. Thus the collection of such isotopy classes of embeddings is naturally isomorphic to $\bbZ \circleplus \bbZ/2\bbZ$. In all but the last Step 3 below, orientations are not of concern.
\end{technical}

I found Gabai's proof, although only five pages long (pages 4--8 of \cite{gabai}), to be a challenge to understand. (Indeed he offers two proofs, the second using a different technique and substantially longer, which I haven't read.) So, in the hope of making easier others' reading of his proof, I offer here my interpretation of it, presented in a manner that is more in the fashion of traditional proofs of this sort. If you wish, this paper may be regarded as my Bourbaki-Seminar-type rendering of Gabai's ingenious construction. The background requirements are quite elementary, primarily some basic differential topology, including some familiarity with Morse theory.

Gabai also presents in \cite{gabai} a number of interesting corollaries, extensions and variations of the above theorem, and also some history. None of this is repeated here, but all are worth reading.

\section{The Proof of the 4D-LBT}

We'll work entirely in the smooth category throughout this paper (although at a certain point we introduce corners for convenience, as in \cite{gabai}). So all embeddings, isotopies, etc., are assumed to be smooth. No further mention of this assumption will be made. Alternatively one could work entirely in the PL manifold world, but then for smooth conclusions one would have to rely on some esoteric PL versus DIFF results in dimension 4.

\begin{aside}\rm 
If one were to write out the following proof in complete (``machine-readable'') detail, it would become a morass of functions, isotopies and their compositions, and of 3D and 4D sets and subsets, whose notational complexity would obscure a very concrete and clear construction. So following Gabai I present the proof in a manner which keeps named functions to a bare minimum, instead describing various motions in more common, descriptive language (that geomorphic topologists use informally all the time). As an example I (like Gabai) have avoided naming the implicit embedding $\phi \colon S^2 \hookrightarrow S^2 \times S^2$ whose image is $\Sigma$, and also I have not explicitly named any isotopy of $\phi$ or its image that ensues. I leave such elaboration of details of exposition to the reader.
\end{aside}

I'll present the proof in three Steps (whose percentages of overall content and difficulty, for a reader comfortable with this sort of topology, are roughly 5\%, 80\% and 15\%, I'd estimate).

 And now we begin the proof of the 4D-LBT.

\subsection*{Step 1} {\bf Isotoping $\Sigma$ to be in Window Position} (my term).
\smallskip

Given $\Sigma$ as in the 4D-LBT, we can assume (by familiar transversality maneuvers, that after a small isotopy) there is a $2$-disc $D^2 \subset S^2$ so that $\Sigma \cap S^2 \times D^2 = \{\point\} \times D^2$. We wish in effect to reparametrize the complement of $D^2$ in $S^2$ to make it rectangular (instead of round). There are any number of ways to describe this simple-but-important change of viewpoint. Here I'll offer mine.

Let {\it window} $W = W^2 \subset S^2 \setdifference D^2$ be a closed rectangular region whose sides are arcs of constant latitude or longitude. Now, thinking of $W$ as being centered diametrically opposite $D^2$,``radially" expand $D^2$ in $S^2$ to make $\interior D^2$ cover $S^2 \setdifference \interior W$.  So now the `interesting' (i.e.\ non-standard) part of $\Sigma$ lies in $S^2 \times \interior W$. From now on we restrict our attention to $S^2 \times W$ and $\Gamma := \Sigma \cap S^2 \times W$.

\begin{figure}[hbt]
\centering\includegraphics[scale = .8]{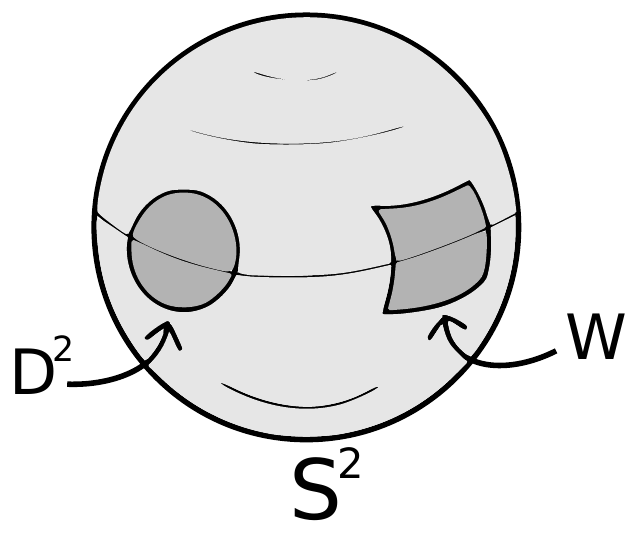}
\caption{The round $D^2 \subset S^2$ and the rectangular window $W^2 \subset S^2$}
\end{figure}

So to prove the 4D-LBT, it suffices to prove:

\begin{theorem}[4D-LBT, Rectangle version] Suppose that $W$ is a closed 2-dimensional rectangle, and that $\Gamma \subset S^2 \times W$ is the image of a boundary-faithful embedding of W which carries some $\nbd(\bdy W)$ identically onto  $\{\point\} \times \nbd(\bdy W)$, where $\nbd(\bdy W)$ is a neighborhood of $\bdy W$ in $W$. Then $\Gamma$ is unknotted, meaning that $\Gamma$ can be isotoped to one of countably-infinitely many canonical positions, indexed by the degree of the projection of $\Gamma$ to the first factor $S^2$.
\end{theorem}

The proof of this theorem is the goal of the remainder of this paper.

\subsection*{Step 2. Isotoping $\Gamma$ to become (the trace of) an isotopy.}

In what follows we parametrize $W$ as $W = I \times J$, where I and J are closed intervals. The following Proposition is the heart of Gabai's proof.

\begin{proposition}[Leveling Proposition \footnote{It seems to me that this fundamental Proposition merits its own name. I'm not all that happy with `Leveling'. Suggestions welcomed.}$ ^{,}$\footnote{This Proposition readily extends to the more general context where one merely assumes that (the embedding) $\bdy \Gamma \subset S^2 \times \bdy(I \times J)$ `respects' the four sides of $\bdy(I \times J)$. That is, at the $I \times \bdy J$ levels, $\bdy \Gamma$ represents two (arbitrary) boundary-faithful embeddings of $I$ into $S^2 \times I$, and along the sides $\bdy I \times J$, $\Gamma$ represents an isotopy of $\bdy I$ in $S^2 \times \bdy I$. In this context the Proposition can be regarded as a codimension-2 version of Concordance (or Pseudoisotopy) implies Isotopy in dimension $4$.}]

Suppose that $I$ and $J$ are compact intervals, and suppose that $\Gamma \subset S^2 \times I \times J$ is the image of a boundary-faithful embedding of $I \times J$ which is standard (the ``identity'') on a neighborhood of $\bdy (I \times J)$ in $ I \times J$. Then $\Gamma$ can be isotoped rel $\bdy \Gamma$ so that $\proj_J | \Gamma$ has no critical points on $\Gamma$, where $\proj_J \colon S^2 \times I \times J \rightarrow J$ is the projection. That is, $\Gamma$ can be isotoped so that $\proj_J | \Gamma \colon \Gamma \twoheadrightarrow J$ is a submersion, and consequently for each $t \in J$, $\Gamma \cap S^2 \times I \times \{t\}$ is a spanning arc in $\Gamma$ joining the two components of $(\bdy I) \times J$.
\end{proposition}

Thus, after repositioning, $\Gamma$ may be regarded as the trace (= image in $S^2 \times I \times J$) of a boundary-faithful isotopy, parametrized by $t \in J$, of $I$ in $S^2 \times I$, constant on $\bdy I$, which begins and ends with the standard inclusion $\{\point\} \times I \hookrightarrow S^2 \times I$.

\begin{proof} The proof will occupy the remainder of Step 2.

We may assume that $\proj_J|\Gamma$ is a Morse function on $\Gamma$, that is, its critical points are nondegenerate. The proof of the Leveling Proposition uses ``codimension-2 embedded Morse theory'' to cancel these critical points against each other, which will be accomplished by isotopy of $\Gamma$ in $S^2 \times I \times J$ and will leave $\Gamma$ in the desired position. (A baby example of such codimension-2 embedded Morse theory is when one repositions a knot in $\bbR^3$ to have some desired nice form with respect to a height function there. An example is putting the knot in ``bridge position.'' Also, I think of the proofs of the smooth $S^1 \subset \bbR^2$ and $S^2\subset \bbR^3$ Schoenflies Theorems as being examples of codimension-1 embedded Morse theory.)

For concreteness we suppose that interval $I = [-1,1]$ and interval $J = [-8,8]$. (This curious choice for $J$ mimics Gabai, enabling that the other key $J$-levels which come up in his argument be integers. Readers averse to all of the negative $J$-integers below may adjust this choice according to their tastes.) We may  assume that there are $p \ge 0$ index-$0$ (local minima) critical points of $\proj_J|\Gamma$, all of which appear at $J$-level $-7$, and $q \ge 0$ index-$2$ (local maxima) critical points in $\Gamma$, all of which appear at $J$-level $7$. Thus there are $p+q$ index-$1$ (saddle) critical points in $\Gamma$, which initially we can assume all appear at $J$-level $0$. We may think of these critical points as providing a handle structure on (say) $\Gamma \rel I \times \{-8\}$  (or, if you prefer, on $\Gamma \cup (I \times [-9,-8]) \rel I \times [-9,-8]$ ). By 1-handle slides (internal in $\Gamma$) we may assume that each of the first $p$ of the $1$-handles attaches (connects) a single $0$-handle to the component of $\Gamma \cap ( S^2 \times I \times [-8,0) )$ which contains no critical points. These $p$ $1$-handles we then can push downward to $J$-level $-3$, and the remaining $q$ $1$-handles we can push upward to $J$-level $3$.

\begin{figure}[hbt]
\centering\includegraphics{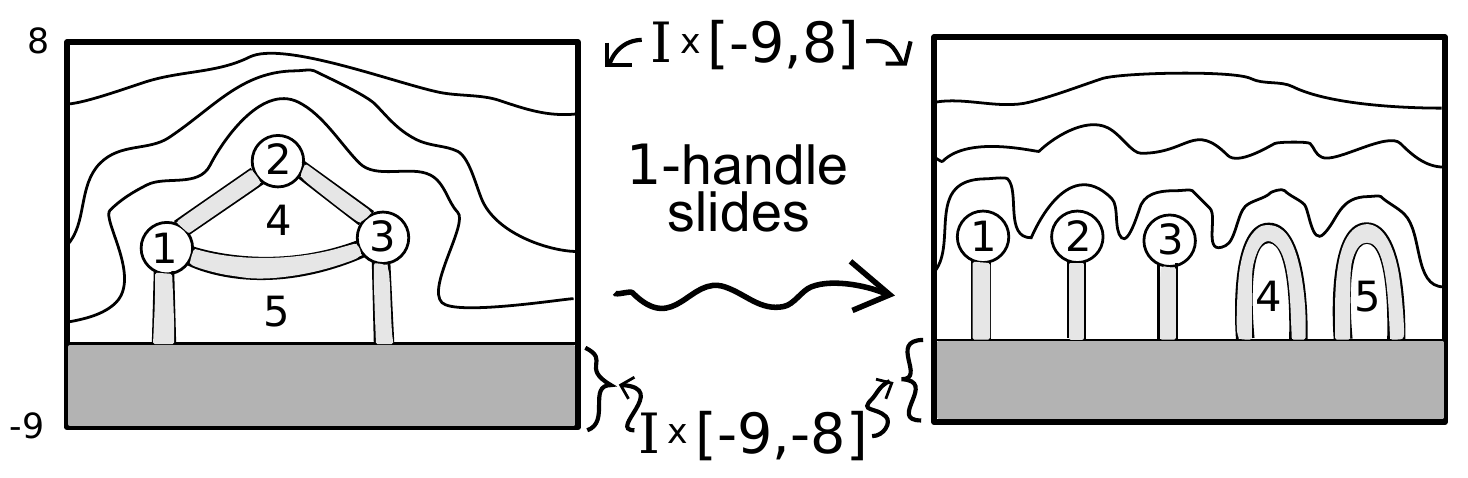}
\caption{Here the discs $1$, $2$ and $3$ represent (local) minima, and the regions $4$ and $5$ represent (local) maxima}
\end{figure}

From now on we'll restrict attention to the subset $\Gamma \cap (S^2 \times I \times [-8,0])$ of $\Gamma$, showing how to cancel the $2p$ critical points there. Once that is done, we can use the same argument (turned upside-down, as usual) to cancel the $2q$ critical points in $\Gamma$ which lie above $J$-level $0$.

We break the remainder of the proof of the Leveling Proposition, that is Step 2, into four subSteps.
\smallskip

\par {\bf{Step $\mathbf{2_0}$. Horizontal-Vertical repositioning of $\Gamma$.}}

We want to recast our picture of $\Gamma$ using ``horizontal-vertical'' positioning, in order to visualize more clearly the handle structure on $\Gamma$. See Figure 3. (This sort of positioning is standard in codimension $\ge 1$ embedded Morse theory. Strictly speaking this means that we're passing to the ``smooth-with-corners'' category, at least for purposes of description. But we're comfortable with that.)

\begin{notation}\rm
 We'll describe certain subsets of our (to-be-repositioned) $\Gamma$ in $S^2 \times I \times [-8,0]$ as being either  ``horizontal'' or ``vertical.'' Such a subset will be of the form $K \times L$, with $K \subset S^2 \times I$ and $L \subset J$, where either
\begin{itemize}
  \item $K \times L$ is ``horizontal,'' meaning that $K$ is a compact $2$-manifold (each component of which is a round $2$-disc representing a $0$-handle, or an embedded rectangle representing a $1$-handle) and $L$ is a point in $J$, or else
  \item $K \times L$ is ``vertical,'' meaning that $K$ is a $1$-manifold (namely a closed interval union a finite number of circles, with $K \scaps S^2 \times \bdy I = \bdy K$) and $L$ is a (nontrivial) subinterval of $J$.
\end{itemize}

Hence all of the key $K$-manifolds which appear below, namely $I_0$, the $D_i$\pluralspace s, the $C_i$\pluralspace s, the $H_i$\pluralspace s, $I_{1\mathendash p}$, $\Omega$ and $\Delta$, etc.,
 \emph{should be regarded as subsets of} $S^2 \times I$. When I want to regard these various sets (generically here denoted $X$) as subsets of $S^2 \times I \times [-8,0]$, I will write them as $X \times \{t\}$ for some specific value of $t \in [-8,0]$. Similarly for their products with $J$-intervals I will write $X \times [s,t]$ (or $X \times (s,t)$, etc.) for $s, t \in [-8,0]$. Indeed, one potentially confusing aspect of the proof is keeping straight where various named subsets of $S^2 \times I \times J$ lie. I hope that my exposition keeps this issue relatively clear. (To contrast the above discussion with Gabai's presentation, we note that he generally regards and describes his various interesting subsets of $\Gamma$ (= $R$, for him) as \emph{always} being subsets of $S^2 \times I \times J$, and he then uses and manipulates their images in $S^2 \times I$ under the projection $\pi \colon S^2 \times I \times J \to S^2 \times I$, to describe his various situations and operations. So I'm choosing to avoid explicitly using $\pi$. Just a matter of taste.)
\end{notation}

\newpage

\begin{figure}[h]
\centering\includegraphics[scale = .85]{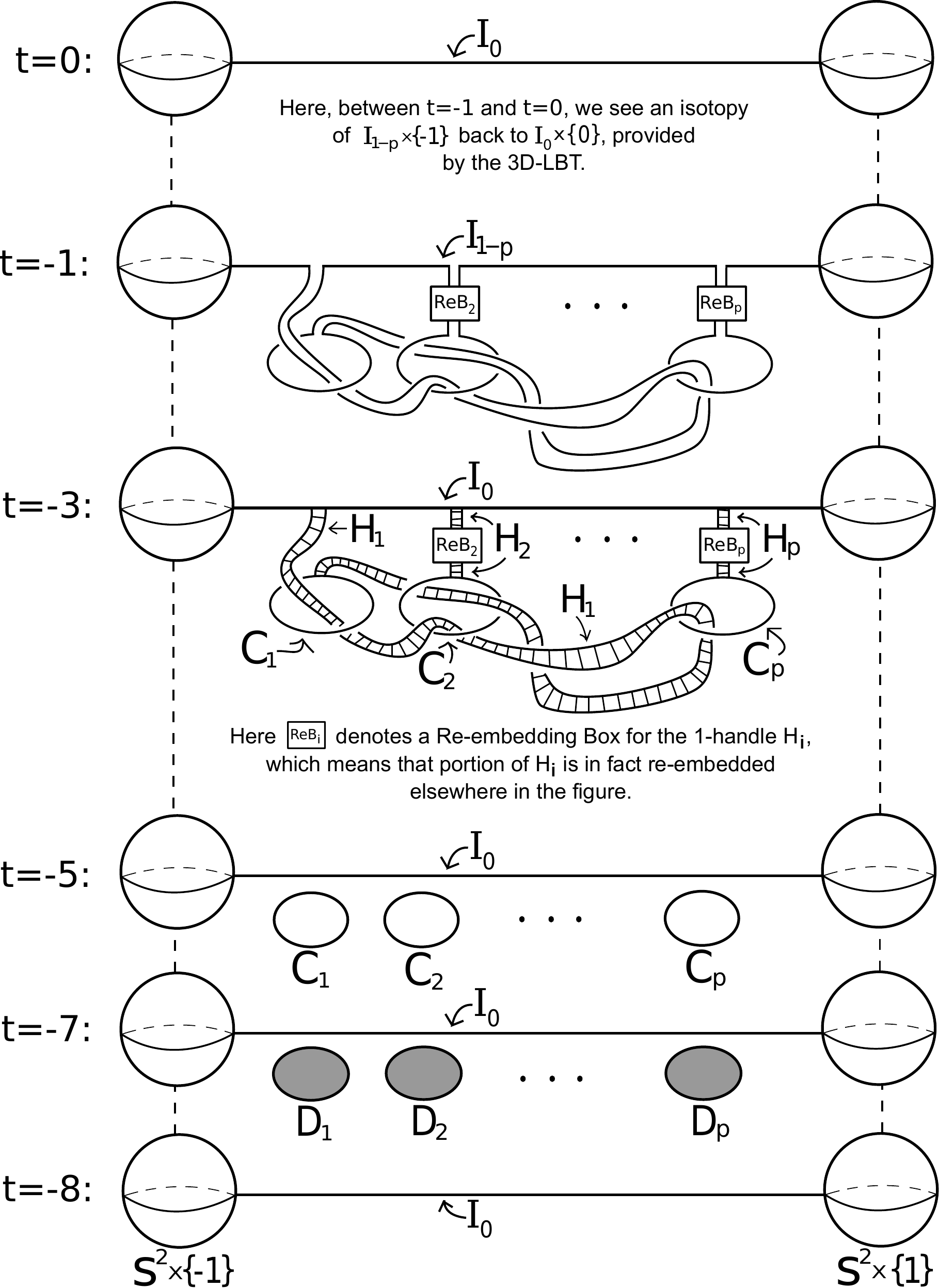}
\caption{The initial horizontal-vertical positioning of $\Gamma \cap S^2 \times I \times [-8,0]$. Here we show several 3-dimensional slices $\Gamma \cap S^2 \times I \times \{t\}$, for $t = -8, -7, -5, -3 ,-1$ and $0$.}
\end{figure}

\newpage

In this paragraph we describe our desired repositioning of (the portion of) $\Gamma$ in $S^2 \times I \times [-8,0]$, during which we leave fixed $\Gamma \scaps S^2 \times \bdy(I \times [-8,0])$, in the end putting $\Gamma \cap S^2 \times I \times [-8,-1]$ into horizontal-vertical position (see Figure 3). Working upwards in stages from the bottom $J$-level $-8$, we begin by arranging that $\Gamma$ is vertical between $J$-levels $-8$ and $-7$, that is, $\Gamma \scaps S^2 \times I \times [-8,-7) = I_0 \times [-8,-7)$, where $I_0  = \{\point\} \times I \subset S^2 \times I$ is a straight arc joining the two boundary spheres of $S^2 \times I$.  Then at $J$-level $-7$, we arrange to appear there $p$ disjoint flat round (and small, if you wish) 2-discs  $D_i \times \{-7\}$, $1 \le i \le p$, representing the $0$-handles of $\Gamma$, that is (neighborhoods of) the $p$ local-minimum critical points in $\Gamma$. For convenience and concreteness we imagine them as appearing neatly strung out alongside of (and disjoint from) $I_0 \times \{-7\}$. Next we can arrange that $\Gamma$ be vertical between $J$-levels -7 and -3, that is $\Gamma \scaps S^2 \times I \times (-7,-3) = (I_0 \scups \bigcup_{1 \le i \le p} C_i) \times (-7,-3)$, where $C_i = \bdy D_i$. Then at $J$-level $-3$ there appear $p$ embedded disjoint $2$-dimensional rectangular bands $H_i \times \{-3\}$, $1 \le i \le p$, representing the initial $p$ index-$1$ critical points in $\Gamma$ mentioned above, with band $H_i$ joining circle $C_i$ to arc $I_0$. Indeed each band $H_i$ we picture as a (possibly complicated looking knotted, twisting) rectangle parametrized as $H_i$ = $A_i \times B_i$, where arc $A_i$ joins $C_i$ to $I_0$ (these $A_i$\pluralspace s we think of as being knotted and linking through the $C_j$\pluralspace s with abandon), and the transversal arc $B_i$ is very short. (It also helps to think of each individual disc-arc pair $D_i$ and $H_i \cap I_0$ as being small and adjacent, whereas these different disc-arc pairs are relatively far apart, strung out along $I_0$ in $i$-order.) Next we can arrange that $\Gamma \cap S^2 \times I \times (-3,-1] = I_{1 \text{-} p} \times (-3,-1]$, where $I_{1\mathendash p}$ is the spanning arc in $S^2 \times I$ which results from doing $p$ surgeries on $I_0 \cup \bigcup_{1 \le i \le p} C_i$ using $\bigcup_{1 \le i \le p} H_i$. Finally, to finish our repositioning of (the bottom half of) $\Gamma$, we can assume that $\Gamma \cap S^2 \times I \times [-1,0]$ simply represents (the trace of) an isotopy of $I_{1 \text{-} p} \times \{-1\} \subset S^2 \times I \times \{-1\}$ back to $I_0 \times \{0\} \subset S^2 \times I \times \{0\}$, whose existence is justified by the 3D-LBT. (Note that this arbitrary isotopy may sweep over the $S^2$ factor with any possible degree in $\bbZ$, but we don't care; that's for later discussion (in Step 3). In the remainder of Step 2 we'll leave $\Gamma \cap S^2 \times I \times [-1,0]$ unperturbed.

At this point we turn our attention to canceling, via isotopy of $\Gamma$, these $p$ $0$-handles with the $p$ $1$-handles. This will be explained and accomplished in the remaining three subSteps of Step 2, which correspond to the cases $p = 1, p = 2$, and $p \ge 3$.
\smallskip

\par {\bf{Step $\mathbf{2_1}$. Here we suppose that $p = 1$.}}

In this case we note that we could cancel $D_1$ against $H_1$ provided that $H_1 \cap D_1 \subset C_1$, that is, $H_1 \cap \interior{D_1} = \emptyset$. For in that case we could push (by isotopy of $\Gamma$) $D_1 \times \{-7\}$ up to $J$-level $-3$, where $(D_1 \cup H_1) \times \{-3\}$ becomes a cancelable pair, in customary fashion. (Aside: Alternatively in this case we could push $H_1 \times \{-3\}$ down to $J$-level $-7$ and do our canceling there. Take your pick. I'll stick with the first option.)

From now on in this Step $2_1$ we'll drop the subscripts on $D_1, C_1$ and $H_1$, writing them as $D, C$ and $H$.

If $H \cap \interior D \neq \emptyset$ (in $S^2 \times I$, I remind the reader), we will use the following

\begin{observation}\label{observation}\leavevmode
\begin{enumerate}
\item There is a $2$-disc $E \subset S^2 \times \interior I$ such that $E \cap (C \cup H \cup I_0) = \bdy E = C$, {\rm{and}}
\item $D$ is \text{isotopic} to $E \rel C$ in $S^2 \times I \setdifference I_0$.
\end{enumerate}
\end{observation}

\begin{proof}\leavevmode

\begin{enumerate}
\item See Figures 4 and 5. Let $\Omega$ be the $2$-sphere boundary of a natural small (i.e.\ close, or tight) neighborhood of $I_0 \cup (S^2 \times \bdy I)$ in $S^2 \times I$. We may assume that $\Omega \cap H$ is a small transverse ($B$-type) arc of $H$. Then $E$ is obtained from $D$ by pushing small $D$-neighborhoods of the intersection arcs of $H$ with $\interior D$, along $H$ toward $H \cap I_0$, where they can be connect-summed with parallel copies of $\Omega$ to get rid of these unwanted intersections of $H$ with $D$. To elaborate a bit (for newbies), one might want to think of doing this operation one intersection arc of $H \cap \interior D$ at a time, beginning with the one which is closest in $H$ to $H \cap I_0$, then working on the next closest intersection-arc, etc., each time pushing $\interior D$ along a longer-but-smaller-radius tube and using a smaller (tighter) copy of $\Omega$ for the connect-summing.

\item This follows from the facts that $S^2 \times \interior{I} \setdifference I_0$ is homeomorphic to $\bbR^3$, and two (smooth) 2-discs in $\bbR^3$ with common boundaries are isotopic fixing the common boundary. (This latter fact is proved, after isotoping (one of) the discs to make their interiors disjoint near their boundaries, by a standard innermost circle-of-intersection argument. But beware of the ``standard mistake'': If $B$ is a $3$-ball in $\bbR^3$ whose boundary sphere is the union of a subdisc of $D$ and a subdisc of $E$ which intersect only in their common boundary circle, then $B$ may contain $C$, and so $B$ cannot be used to isotope the subdisc of $E$ off of $D$ (or vice versa) {\it{keeping $C$ fixed}}.)
\end{enumerate}

\end{proof}

\newpage

\begin{figure}[t]
\centering\includegraphics[scale = 1]{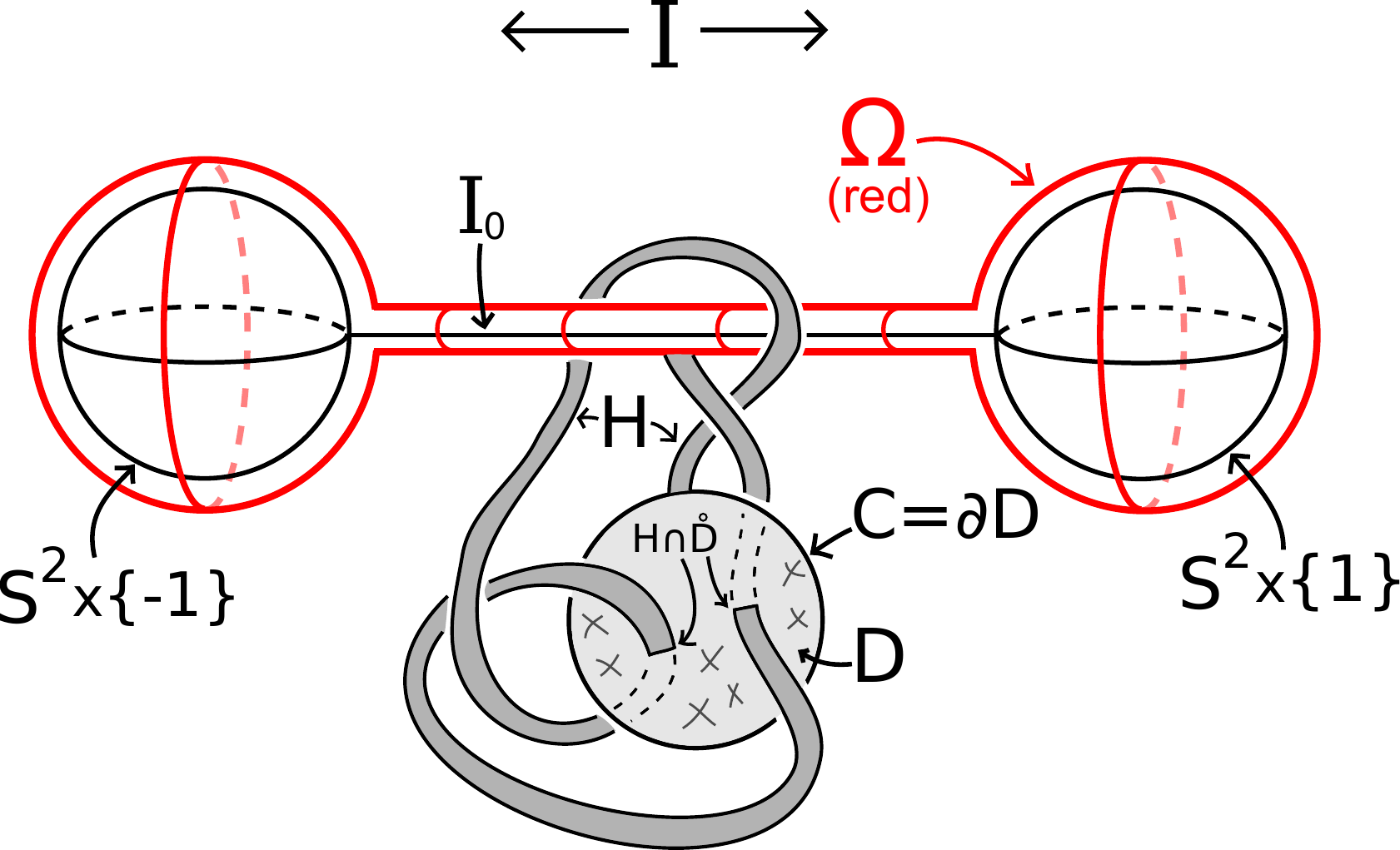}
\caption{The Step $2_1$  subsets $I_0, C, D, H$ and $\Omega$ of $S^2 \times I$}
\end{figure}

\begin{figure}[h]
\centering\includegraphics[scale = 1]{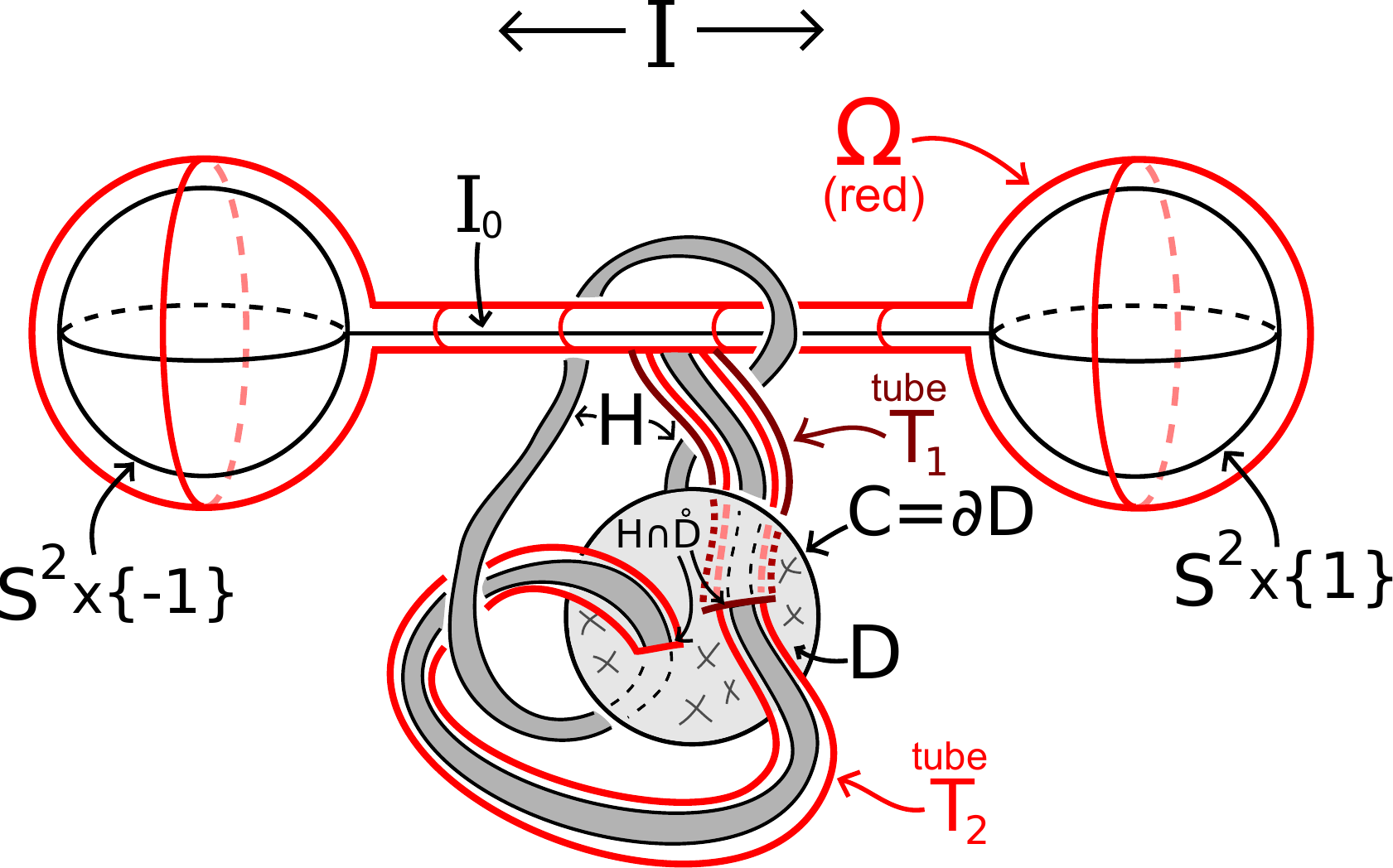}
\caption{Adding (to Figure 4) the tubes (shown as rectangles) $T_1$ and $T_2$}
\end{figure}

\newpage

There is an (arguably) alternative way of proving this Observation 2.3, as follows (the proof above is presented  \`{a} la Gabai):

\begin{claim}
There is an (ambient) isotopy of $S^2 \times I$, fixed on $S^2 \times \bdy I \cup I_0 \cup C$, which moves $H$ off of $\interior D$.
\end{claim}

\begin{proof}
It suffices to prove this with $H$ replaced by a subarc $A$ of $H$, namely $A  = A \times \{\point\} \subset H = A \times B$. Since $S^2 \times \interior{I} \setdifference I_0$ is diffeomorphic to $\bbR^3$, we can transfer our setup to $\bbR^3$. To do so, identify the closed 3-ball that the 2-sphere $\Omega$ (above) bounds in $S^2 \times I \setdifference I_0$, with the standard 3-ball $B^3_2$ of radius 2 in $\bbR^3$, in such a manner that $D$ becomes the standard disc of radius 1 in the $x$-$y$ plane. Let arc $A_0$ be the interval $[1,2]$ in the $x$-axis, which joins $\bdy D$ to $\bdy B^3_2$. Then we can regard $A$ as being obtained from $A_0$ by a reembedding of $A_0$ into $\interior B^3_2 \setdifference \bdy D$ which is standard near $\bdy A_0$. We can isotope $A$ off of $\interior D$ keeping $C \cup \bdy B^3_2$ fixed, by 1) arranging that $A$ has (transverse algebraic) intersection number 0 with $\interior D$, by spinning $A$ around $\bdy D$ at its intersection point with $\bdy D$, and then 2) removing the points of intersection of $A$ with $\interior D$ in successive $\pm$ pairs which are innermost in $A$, noting that this is possible by homotopy and hence regular homotopy of $A$, and then that any self intersections in a regular homotopy of $A$ can be gotten rid of by the 3D light bulb trick, using $\bdy B^3_2$.
\end{proof}

Resuming our logical thread, to complete this Step $2_1$ we perform an isotopy of $\Gamma$ which moves $D \times \{-7\}$ up to become $E \times \{-3\}$, where it can be cancelled with $H \times \{-3\}$. The idea is that, while pushing $D \times \{-7\}$ vertically upward, we perform the isotopy of $D$ to $E$ as we go along. In more detail, let $D_t \subset S^2 \times I$ denote the image of $D$ under the isotopy of Observation \ref{observation}, with (for notational convenience) $t$ running from $-7$ to $-3$, so that $D_{-7} = D$ and $ D_{-3} = E$. Then, for $t \in [-7,-3]$, let $\Gamma_t := (\Gamma \setdifference (D \times \{-7\} \cup C \times [-7,t)\,)) \cup D_t \times \{t\}$, which describes an isotopy of $\Gamma_{-7}$ to $\Gamma_{-3}$. The final position $\Gamma_{-3}$ of $\Gamma$ is vertical everywhere in $S^2 \times I \times [-8,-1]$ except at $J$-level $-3$, where $(E \cup H) \times \{-3\} \subset \Gamma_{-3}$ is in cancellable position. End of Step $2_1$.

We continue the proof of the Leveling Proposition 2.2, proceeding now to \smallskip

\par {\bf{Step $\mathbf{2_2}$. Here we suppose that $p = 2$.}}

In this case we note that if $H_1 \cap D_2 = \emptyset$ (in $S^2 \times I$, recall), then we can, by isotopy of $\Gamma$, push $H_1 \times \{-3\}$ down to to become $H_1 \times \{-6\}$, and push $D_2 \times \{-7\}$ up to become $D_2 \times \{-4\}$. Then we can cancel $D_1 \times \{-7\}$ against $H_1 \times \{-6\}$ using Step $2_1$, applied say in the region $\Gamma \cap S^2 \times I \times [-8,-5]$, and independently we can cancel $D_2 \times \{-4\}$ against $H_2 \times \{-3\}$ again using Step $2_1$, now applied say in the region $\Gamma \cap S^2 \times I \times [-5, -2]$. So, to accomplish this Step $2_2$ it suffices to prove

\begin{lemma}[Disjointness Lemma] By isotopy of $\Gamma \cap S^2 \times I \times [-4,-1] \rel \Gamma \cap S^2 \times \bdy(I \times [-4,-1])$ we can arrange that $H_1 \cap D_2 = \emptyset$. More precisely, after this repositioning $\Gamma \scaps S^2 \times I \times [-4,-1]$ will be in horizontal-vertical position, having (as at the start) two horizontal components, both at the same $J$-level $-3$, namely disjoint $1$-handles $H'_1 \times \{-3\}$ and (the original) $H_2 \times \{-3\}$, where $H'_1 \cap (I_0 \cup C_1) = H_1 \cap (I_0 \cup C_1)$, such that $H'_1 \cap D_2 = \emptyset$.
\end{lemma}

\begin{proof}
Figures 6 through 9 are meant to illustrate the following discussion. The desired repositioning of $H_1 \times \{-3\}$ in $S^2 \times I \times [-4,-1]$ will involve some vertical (= $J$-direction) motion. To begin, we isotope $\Gamma$ so as to move $H_1 \times \{-3\}$ vertically upward to become $H_1 \times \{-2\}$, which connects $C_1 \times \{-2\}$ to $I_2 \times \{-2\}$, where $I_2 \subset S^2 \times I$ is the spanning arc gotten by doing surgery on $I_0 \cup C_2$ using $H_2$. So now $\Gamma \cap S^2 \times I \times (-3,-2) =  I_2 \times (-3,-2)$. Note that in $S^2 \times I$ the arc $I_2$ has a transverse embedded $2$-sphere, call it $\Delta$ (see Figure 8), which intersects $I_2$ transversely in a single point lying somewhere in the open arc $(I_2 \cap C_2)^{\circ}$. Also we can choose $\Delta$ so that it does not intersect $H_1$ (but it may intersect $D_1$, and it will intersect both $D_2$ and $H_2$). To describe $\Delta$ more precisely, let $I'_2$ be a subarc of $I_2$ which joins some point in $(I_2 \cap C_2)^{\circ}$ to the endpoint $1$ of $I_2$, so that $I'_2 \cap H_1 = \emptyset$. Then $\Delta$ is the boundary of a small regular neighborhood of $I'_2 \scups S^2 \times \{1\}$ in $S^2 \times I$.

Now in $S^2 \times I$ we can isotope $H_1$ off of $D_2$, producing $H'_1$, by pushing a small regular neighborhood $N$ in $H_1$ of the union of the arcs $H_1 \cap \interior{D_2}$, moving each component of $\interioroperator_{H_1}{N}$ keeping $\frontier_{H_1}{N}$ fixed, along a path in $D_2$ toward and past $C_2 = \bdy D_2$, using $\Delta$ to reroute this isotopy to make the moving $H_1$ never intersect $I_2$. (We note that the full isotopy-image of $H_1$ will intersect $\interior{D_1}\cup \interior{D_2}$ in every component of $\Delta \cap (\interior{D_1} \cup \interior D_2)$, including the circles. Not an issue.) Hence we may assume that we have an associated covering ambient isotopy of $S^2 \times I$ which {\it{leaves $I_2$ fixed}} (that's a key point!).

\newpage

 \begin{figure}[th]
\centering\includegraphics[scale = .55]{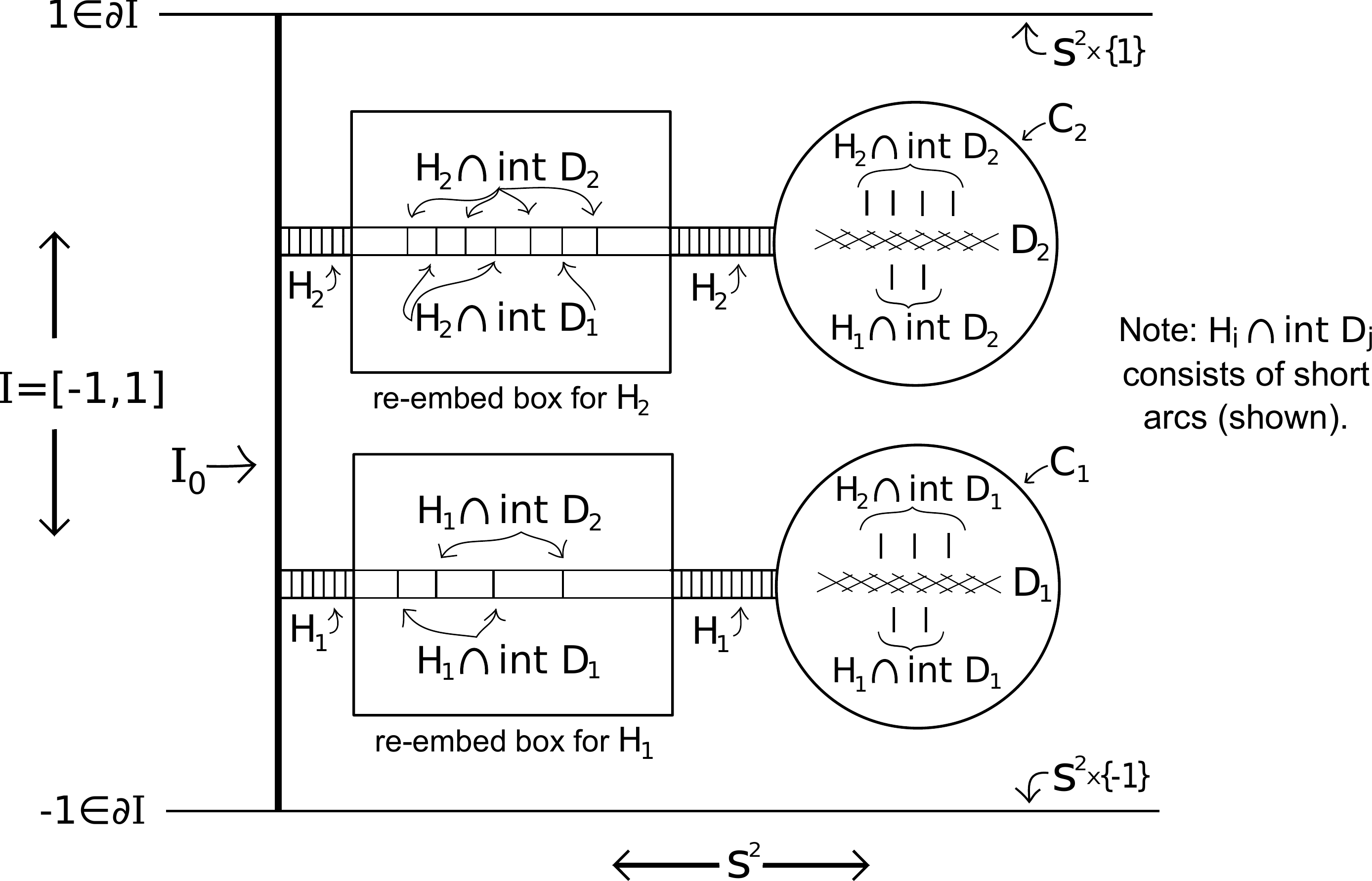}
\caption{The Step $2_2$ subsets $I_0, D_1, D_2, H_1$ and $H_2$ of $S^2 \times I$}
\end{figure}

\begin{figure}[bh]
\centering\includegraphics[scale = .55]{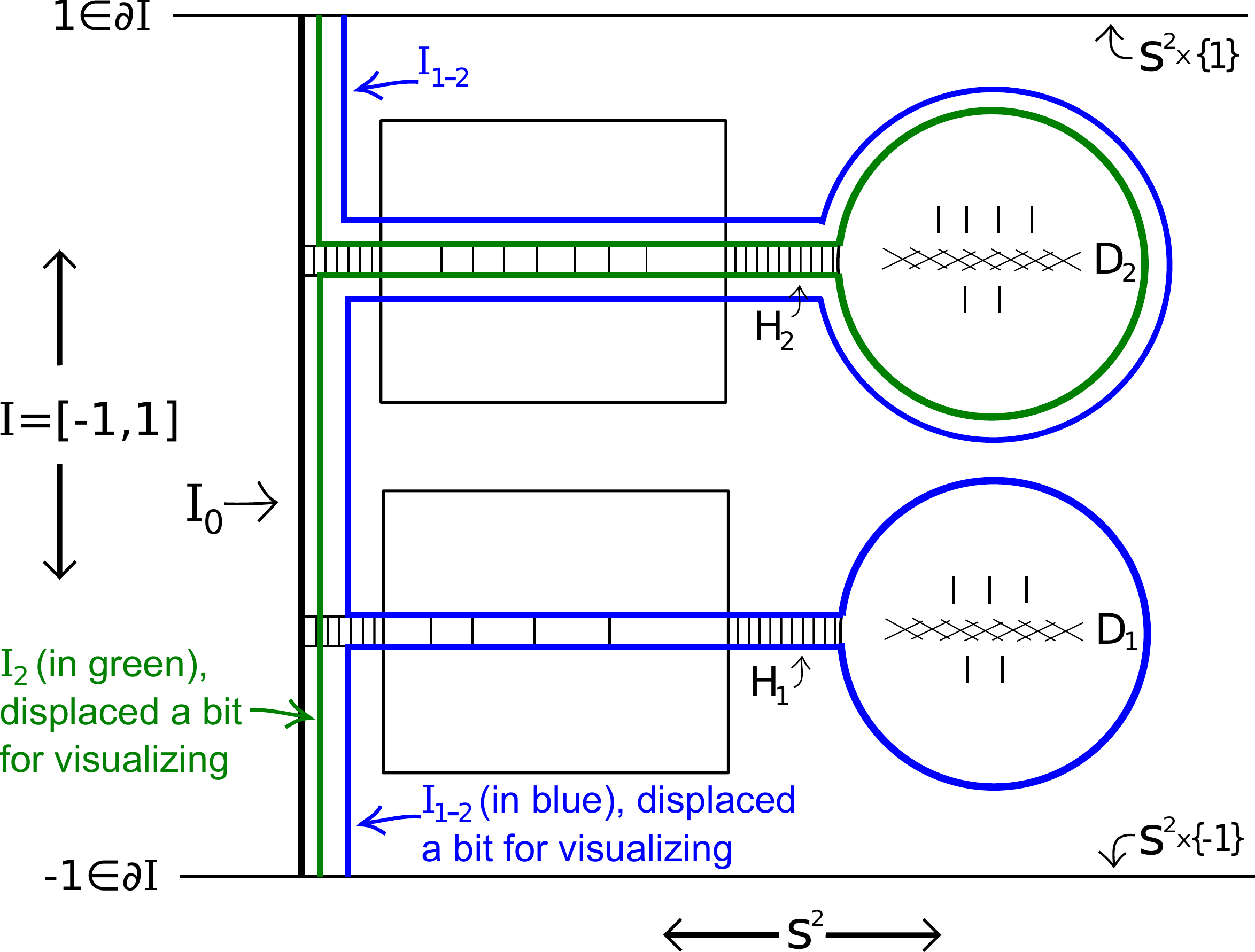}
\caption{Adding (to Figure 6) the subsets $I_{1\mathendash 2}$ and $I_2$ of $S^2 \times I$}
\end{figure}

\newpage

\begin{figure}[th]
\centering\includegraphics[scale = .55]{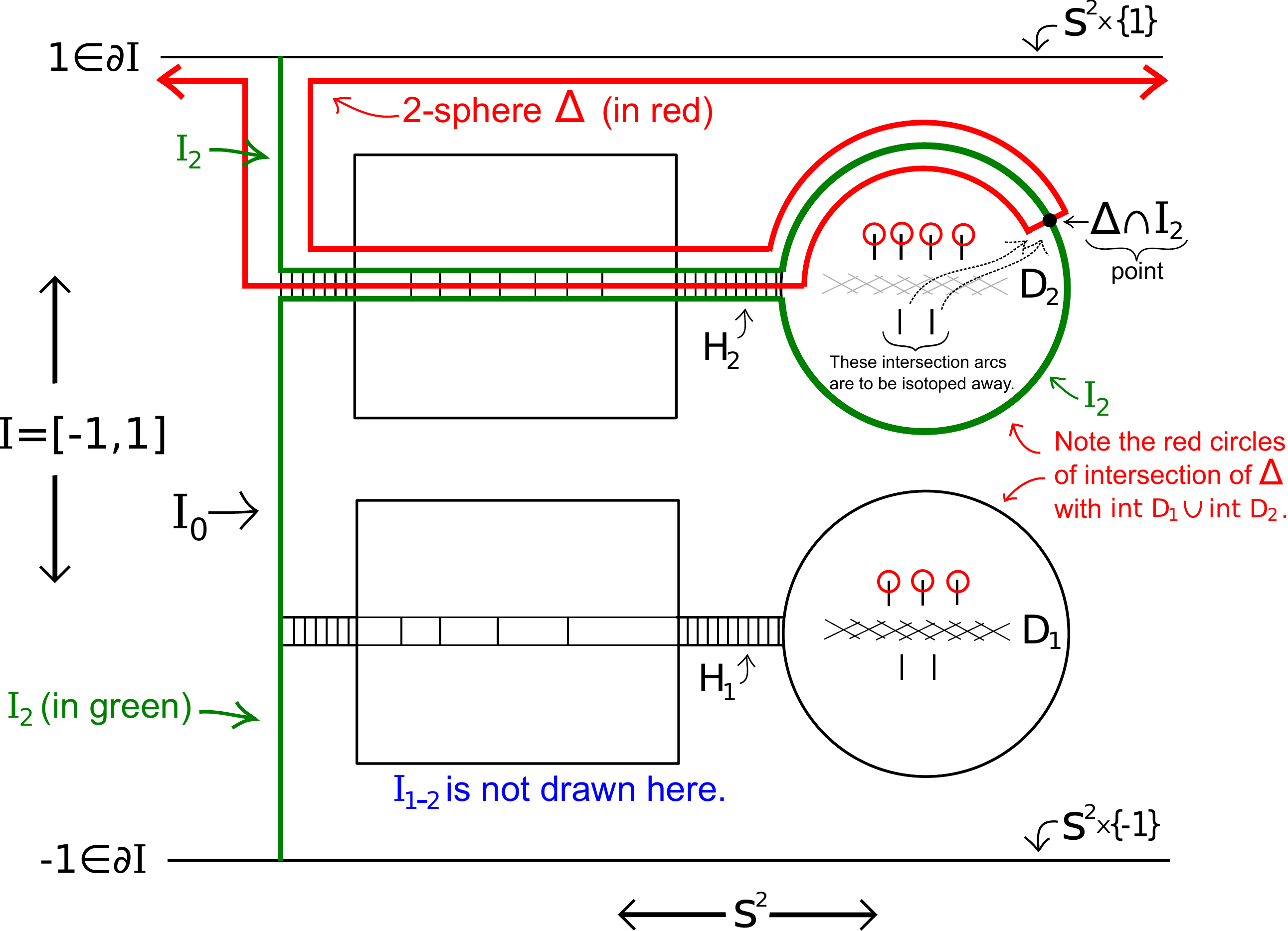}
\caption{Adding (to Figure 7) the $2$-sphere $\Delta$ in $S^2 \times I$ }
\end{figure}

\begin{figure}[bh]
\centering\includegraphics[scale = .55]{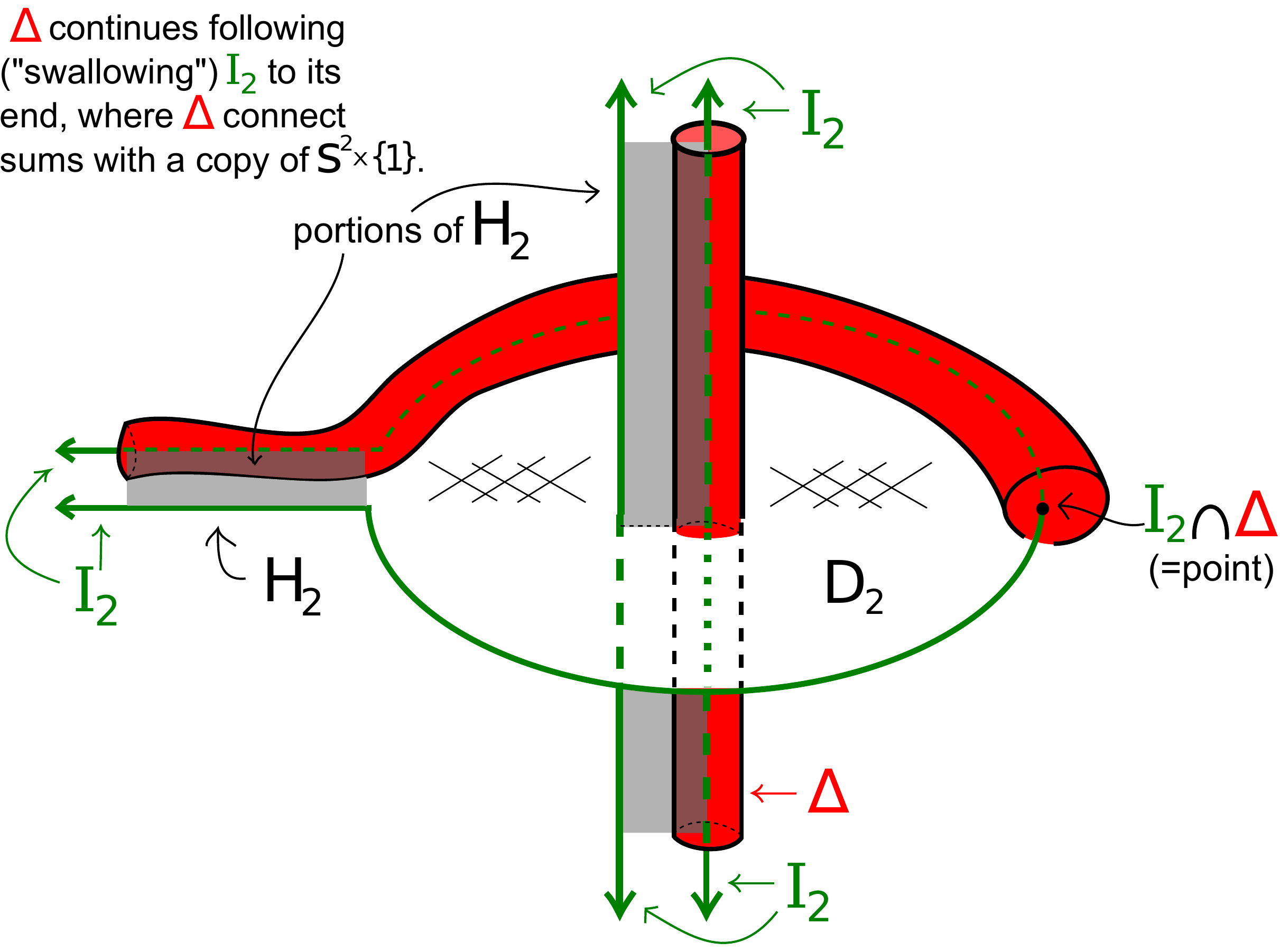}
\caption{A more accurate picture, in $S^2 \times I$, of $D^2$ and portions of $H_2$, $I_2$ (green) and $\Delta$ (red), in Step $2_2$ }
\end{figure}

\newpage

Performing this isotopy in-on-at the level $S^2 \times I \times \{-2\}$, it has a natural extension (by a standard time-parameter-damping scheme) to a $J$-level-preserving ambient isotopy of $S^2 \times I \times [-3,-1]$, fixing the end $J$-levels $-3$ and $-1$. This extension isotopy leaves fixed the portion of the current $\Gamma$ lying in $S^2 \times I \times [-3,-2)$ (recalling that $\Gamma \cap S^2 \times I \times (-3,-2) = I_2 \times (-3,-2)$ ), whereas the portion of $\Gamma$ in $S^2 \times I \times [-2,-1]$ is changed by a $J$-level-preserving isotopy, and so there $\Gamma \cap S^2 \times I \times (-2,-1]$ becomes the track of an isotopy of an arc. Now, to finish, the newly positioned $H'_1 \times \{-2\} \subset \Gamma$ can be pushed back downward vertically by isotopy of $\Gamma$ to become $H'_1 \times \{-3\}$ in $S^2 \times I \times \{-3\}$, where we note that it misses $D_2 \times \{-3\}$. (Note: the interior of the latter disc is {\it{not}} in $\Gamma$.) Hence the desired repositioning of the Disjointness Lemma has been achieved, completing its proof. 
\end{proof}

\begin{exercise}\rm
Instead of the above isotopy, note that one could try to isotope $H_1$ off of $\interior{D}_2$ by pushing the intersection arcs off of $D_2$ at $D_2 \cap H_2$, and then continue pushing them along $H_2$ and across $H_2 \cap I_0$ (everything being done at $J$-level $-2$ here). Why doesn't this isotopy work?
\end{exercise}

This completes Step $2_2$.
\smallskip

\par {\bf{Step $\mathbf{2_3}$. Here we let $p \ge 3$ be arbitrary.}}

This subStep is a routine extension of Step $2_2$, with a couple of minor points meriting attention. The desired analogue of the preceding Disjointness Lemma is

\begin{lemma}[Multiple Disjointness Lemma]
By isotopy of $\Gamma \scaps S^2 \times I \times [-4,-1] \rel \Gamma \cap S^2 \times \bdy(I \times [-4,-1])$ we can arrange that  $H_1 \cap D_i = \emptyset$ for each $i$,  $2 \le i \le p$. More precisely, after this repositioning $\Gamma \scaps S^2 \times I \times [-4,-1]$ will be in horizontal-vertical position, having (as at the start) $p$ horizontal components, all at the same $J$-level $-3$, namely pairwise disjoint $1$-handles $H'_1 \times \{-3\}$ and (the original) $H_i \times \{-3\}$, $2 \le i \le p$, where $H'_1 \cap (I_0 \cup C_1) = H_1 \cap (I_0 \cup C_1)$, such that $H'_1 \cap D_i = \emptyset$ for each $i$,  $2 \le i \le p$.
\end{lemma}

\begin{proof}
We mimic the proof of the preceding Lemma. To make our choices easier and the notation simpler, we will assume (as suggested earlier) that the pairwise disjoint $H_i$-attaching arcs $H_i \cap I_0$, for $1 \le i \le p$, appear in $I_0$ \emph{in increasing $i$-order}, with $H_1 \cap I_0$ being closest to the endpoint $-1$ of $I_0$ and with $H_p \cap I_0$ being closest to the endpoint $1$ of $I_0$. To begin the construction of the proof, isotope $\Gamma$ (as before) so as to move $H_1 \times \{-3\}$ vertically upward to become $H_1 \times \{-2\}$, which connects $C_1 \times \{-2\}$ to $I_{2\mathendash p} \times \{-2\}$, where $I_{2\mathendash p} \subset S^2 \times I$ is the spanning arc gotten by doing $p-1$ surgeries on $I_0 \cup \bigcup_{2 \le i \le p} C_i$ using $\bigcup_{2 \le i \le p} H_i$. So now $\Gamma \cap S^2 \times I \times (-3,-2) =  I_{2\mathendash p} \times (-3,-2)$. Note that in $S^2 \times I$ the arc $I_{2\mathendash p}$ has $p-1$ transverse embedded $2$-spheres, called $\Delta_i$ for $2\le i \le p$, with $\Delta_i$ intersecting $I_{2\mathendash p}$ transversely in a single point which lies somewhere in the open arc $(I_{2\mathendash p} \cap C_i)^{\circ}$. Also we can choose each $\Delta_i$ so that it does not intersect $H_1$ (or $\bigcup_{1 \le j < i} H_i$, for that matter. But $\Delta_i$ may intersect $\bigcup_{1 \le j < i}  D_j$, and it will intersect each $D_j$ and each $H_j$ for $i < j \le p$.) However, we do {\it{not}} choose these $2$-spheres to be disjoint (although such disjoint spheres exist), for reasons that will become clear. To describe each $\Delta_i$ more precisely, let $I'_i$ be a subarc of $I_{2\mathendash p}$ which joins some point in $(I_{2\mathendash p} \cap C_i)^{\circ}$ to the endpoint $1$ of $I_{2\mathendash p}$. Note that $I'_i \cap H_1 = \emptyset$ (indeed $I'_i \cap \bigcup_{1 \le j < i} H_j = \emptyset$), for that matter) and $I'_2 \supset I'_3 \supset\ldots \supset I'_p$. Then we let $\Delta_i$ be the boundary of a small $\epsilon_i $-regular neighborhood of $I'_i \cup S^2 \times \{{1}\}$ in $S^2 \times I$, with $0 < \epsilon_2 < \epsilon_3 <\ldots < \epsilon_p$. These latter inequalities mean that the $\Delta_i$\pluralspace s will intersect, as we now explain. Express each sphere $\Delta_i$ as the union of two closed $2$-discs, $\Delta_i = \Delta'_i \cup \Delta''_i$ whose interiors are disjoint, where $\Delta'_i$ is an $\epsilon_i$-radius $2$-disc which is transverse to $(I_{2\mathendash p} \cap C_i)^{\circ}$, with say the centerpoint of $\Delta'_i$ coinciding with the midpoint of $I_{2\mathendash p} \cap C_i$. Then $\Delta''_i := \Delta_i \setdifference \interior{\Delta}'_i$ . Note that, for $2 \le i < j \le p$, we have that $\Delta_i \cap \Delta_j = \interior{\Delta}''_i \cap \interior \Delta'_j$, and this intersection is an $\epsilon_i$-radius transverse linking circle to $(I_{2\mathendash p} \cap C_j)^{\circ}$.

Now in $S^2 \times I$ we can isotope $H_1$ off of $\bigcup_{2 \le i \le p} D_i$, just as before, to produce $H'_1$. We'll describe this process for an individual $D_i$, noting that all $p-1$ of these upcoming isotopies can be chosen to have pairwise disjoint support, and so they can be performed (composed) either sequentially or simultaneously, to produce the desired isotopy of $H_1$ to $H'_1$. Fix $i \in \{2, 3,\ldots, p\}$. As before, now we push a small regular neighborhood $N_i$ in $H_1$ of the arcs $H_1 \cap \interior D_i$, moving each component of $N_i$ along a path in $D_i$  toward and past $C_i = \bdy D_i,$ using $\Delta''_i$ to reroute this isotopy to make the moving $H_1$ never intersect $I_{2\mathendash p}$. (We note that this isotopy-image of $H_1$ will intersect $\bigcup_{1 \le j \le p} \interior D_j$ in every component of $\Delta''_i \cap \bigcup_{1 \le j \le p} \interior D_j$, including the circles. Again, not an issue.) Hence we may assume that we have an associated covering ambient isotopy of $S^2 \times I$ which{ \it{leaves $I_{2\mathendash p}$ fixed}} (again, that's a key point!).

Performing all of these isotopies (one for each $i \in \{2, 3,\ldots, p\}$) on the level $S^2 \times I \times \{-2\}$, this combined isotopy has a natural extension (as in Step $2_2$) to a $J$-level-preserving ambient isotopy of $S^2 \times I \times [-3,-1]$, fixing the end $J$-levels $-3$ and $-1$. This extension isotopy leaves fixed the portion of the current $\Gamma$ lying in $S^2 \times I \times [-3,-2)$ (recalling that $\Gamma \cap S^2 \times I \times (-3,-2) = I_{2\mathendash p} \times (-3,-2)$ ), whereas the portion of $\Gamma$ in $S^2 \times I \times [-2,-1]$ is changed by a $J$-level-preserving isotopy, and so there $\Gamma \cap S^2 \times I \times (-2,-1]$ becomes the track of an isotopy of an arc. Now, to finish, the newly positioned $H'_1 \times \{-2\} \subset \Gamma$ can be pushed vertically back downward by isotopy of $\Gamma$ to become $H'_1 \times \{-3\}$ in $S^2 \times I \times \{-3\}$, where we note that it misses each $D_i \times \{-3\}$, for $2 \le i \le p$. (Note: the interiors of the latter discs are {\it{not}} in $\Gamma$.) Hence the desired repositioning of the Multiple Disjointness Lemma has been achieved, completing its proof.
\end{proof}

As before, after applying this MD Lemma we can, by isotopy of $\Gamma$, push $H'_1 \times \{-3\}$ vertically downward to to become $H'_1 \times \{-6\}$, and push each $D_i \times \{-7\}$, for $2 \le i \le p$, vertically upward to become $D_i \times \{-4\}$. Then we can cancel $D_1 \times \{-7\}$ against $H'_1 \times \{-6\}$ using Step $2_1$ (applied say in the region $\Gamma \cap S^2 \times I \times [-8,-5]$), and by induction on $p$ we can cancel the remaining $D_i \times \{-4\}$ against the remaining $H_i \times \{-3\}$ (now working say in the region $\Gamma \cap S^2 \times I \times [-5, -2]$).

This completes Step $2_3$, and hence Step 2 and the proof of the Leveling Proposition.
\end{proof} 

\par {\bf{Interlude and Summary.}}

Before proceeding to Step 3 we offer here a summary of the overall proof of the 4D-LBT in terms of homotopy groups. In the following, let $\Embeds^+_{\pitchfork}[S^2, S^2 \times S^2]$ denote the space of smooth embeddings $\{\phi \colon S^2 \hookrightarrow S^2 \times S^2\}$ which, for some fixed basepoint $(z_1,z_2) \in S^2 \times S^2$, satisfy that $\phi(S^2) \cap S^2 \times \{z_2\} = \{(z_1,z_2)\}$, transversally with positive intersection number. Let $\Embeds_{\bdy\textsl{std}}[I \times J, S^2 \times I \times J]$ denote the space of smooth boundary-faithful embeddings $\{\psi \colon I \times J \hookrightarrow S^2 \times I \times J\}$ of a product $I \times J$ of two closed intervals, which carry some neighborhood $N$ of $\bdy (I \times J)$ in $I \times J$ ``identically'' onto $\{w_0\} \times N$ for some distinguished $w_0 \in S^2$.
Similarly let $\Embeds_{\bdy\textsl{std}}[I, S^2 \times I]$ denote the space of smooth boundary-faithful embeddings $\{\psi \colon I \hookrightarrow S^2 \times I\}$ of a closed interval $I$ which carry some neighborhood $N$ of $\bdy I$ in $I$ ``identically" onto $\{w_0\}\times N$.

Using these spaces the proof of the 4D-LBT can be summarized as follows:
\medskip

$\begin{array}{rll}
\pi_0(\Embeds^+_{\pitchfork}[S^2, S^2 \times S^2])  &\isom \pi_0(\Embeds_{\bdy\textsl{std}}[I \times J,S^2 \times I \times J])&\text{(by Step 1)}\\
& \isom \pi_1(\Embeds_{\bdy\textsl{std}}[I, S^2 \times I])&\text{(by Step 2)}\\
& \isom \pi_2(S^2).&\text{(by Step 3)}
\end{array}$

We now turn to:

\subsection*{Step 3. Computing $\bold{\pi_1(\Embeds_{\bdy\textsl{std}}[I,S^2 \times I])}$.}

Here we use a standard fibration-type argument to show that $\pi_1(\Embeds_{\bdy\textsl{std}}[I,S^2 \times I]) \isom \pi_2(S^2)$. We note that Gabai, for this Step, appeals to a deep and demanding result of Allen Hatcher, that $\Diffeo(I \times S^2 \rel \bdy)$ is homotopy equivalent to $\Omega (O(3))$. That's acceptable, but massive overkill imho. Gabai only needs Hatcher's result on the level of $\pi_1$, which the following classical argument shows.

Let $I  \vee_{\bdy}B^3$ denote the wedge of a closed interval $I$ with a compact $3$-ball $B^3$, with the wedge point being in the boundary of each space, and let $e_0$ denote the non-wedge boundary point of $I$. As a specific model we let $I \vee_{\bdy} B^3 = [-3,-1] \times \{(0,0)\} \cup B^3_1 \subset \bbR^3$, where $B^3_r$ denotes the closed round ball in $\bbR^3$ of radius $r$ centered at the origin. Then $e_0 = (-3,0,0)$. Let $\Embeds^+_{\bdy\textsl{std}}[I \vee_{\bdy} B^3, B^3 = B^3_3]$ denote the space of smooth embeddings $\{\mu \colon I \vee_{\bdy} B^3 \hookrightarrow B^3_3\}$ which are orientation-preserving on $B^3$ and are standard (i.e.\ the identity-inclusion) on some neighborhood of $e_0$, such that $\mu^{-1}(\bdy B^3_3) = \{e_0\}$. Let $\Embeds^+[B^3, \interior B^3]$ denote the space of orientation-preserving smooth embeddings. Then we have a natural ``forgetful" map $\Embeds^+_{\bdy\textsl{std}}[I \vee_{\bdy} B^3, B^3] \twoheadrightarrow \Embeds^+[B^3, \interior B^3]$ gotten by restriction to the source $B^3$. This map is a fibration, with the fiber being our space of interest $\Embeds_{\bdy\textsl{std}}[I,S^2 \times I]$. (We note that, alternatively, and perhaps making more transparent the upcoming remarks, we could restrict the embeddings of the $B^3$\pluralspace s in these total and base spaces to be scaled isometries, that is, isometries multiplied by a varying scalar $s > 0$.)

\begin{figure}[hb]
\centering\includegraphics[scale = .7]{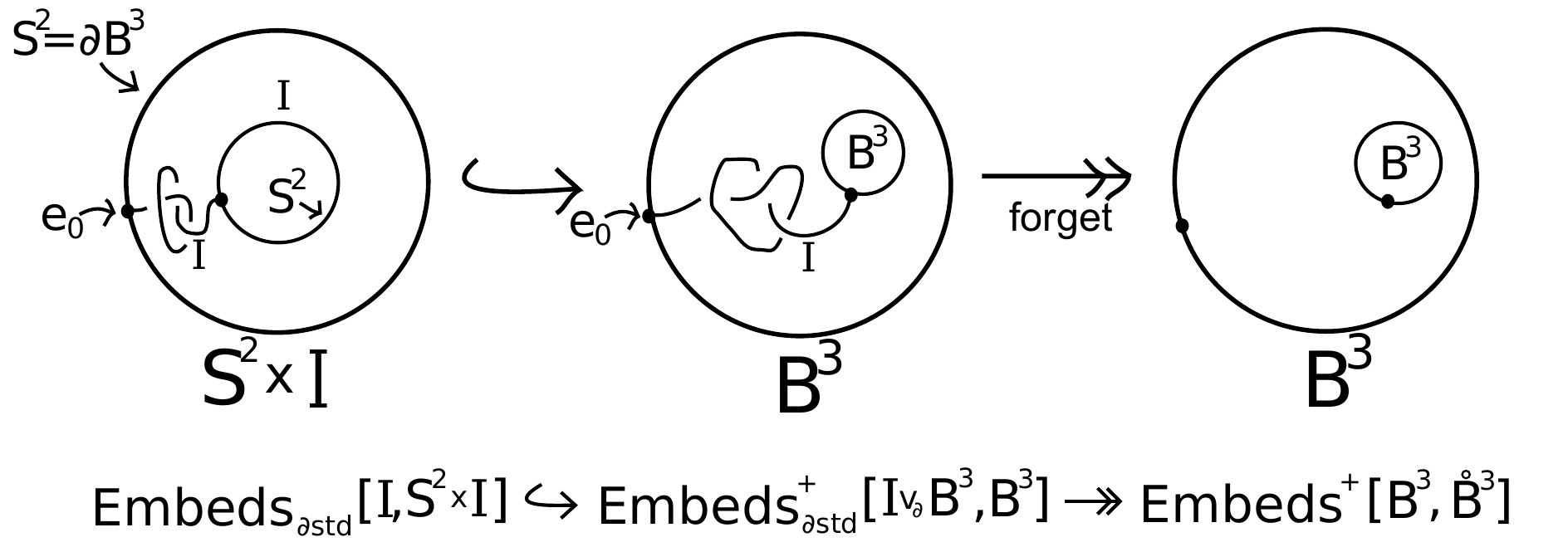}
\caption{The fibration of Step 3}
\end{figure}

\newpage

From this fibration we obtain an exact sequence of homotopy groups:

\begin{align*}
\ldots\rightarrow &\pi_2(\Embeds^+[B^3, \interior B^3]) \rightarrow \pi_1(\Embeds_{\bdy\textsl{std}}[I,S^2 \times I]) \rightarrow\pi_1(\Embeds^+_{\bdy\textsl{std}}[I \vee_{\bdy} B^3, B^3]) \rightarrow\\
&\pi_1(\Embeds^+[B^3, \interior B^3]) \rightarrow\pi_0(\Embeds_{\bdy\textsl{std}}[I,S^2 \times I]) \rightarrow\ldots
\end{align*}
By well-known arguments in smooth topology we know that $\Embeds^+_{\bdy\textsl{std}}[I \vee_{\bdy} B^3, B^3]$ has the homotopy type of $SO(2) \isom S^1$ and $\Embeds^+[B^3, \interior B^3]$ has the homotopy type of $SO(3) \isom \bbR P^3$, and the forgetful map above corresponds to the natural inclusion $SO(2) \hookrightarrow SO(3)$. Hence the above sequence of homotopy groups becomes
\begin{align*}
\ldots\rightarrow& \pi_2(SO(3)) \rightarrow
\pi_1(\Embeds_{\bdy\textsl{std}}[I,S^2 \times I]) \rightarrow \pi_1(SO(2)) \rightarrow\\
&\pi_1(SO(3)) \rightarrow \pi_0(\Embeds_{\bdy\textsl{std}}[I,S^2 \times I])\rightarrow\ldots.
\end{align*}
Evaluating these groups we obtain the exact sequence
$$ \ldots \rightarrow 0 \rightarrow \pi_1(\Embeds_{\bdy\textsl{std}}[I,S^2 \times I]) \rightarrow \bbZ \twoheadrightarrow \bbZ/2 \rightarrow 0 \ldots, $$
which implies that $\pi_1(\Embeds_{\bdy\textsl{std}}[I,S^2 \times I]) \isom \bbZ$.

We get more information if we compare directly the two fibration sequences
$$ \Embeds_{\bdy\textsl{std}}[I,S^2 \times I]) \hookrightarrow \Embeds^+_{\bdy\textsl{std}}[I \vee_{\bdy} B^3, B^3] \twoheadrightarrow \Embeds^+[B^3, \interior B^3] $$
with the well-known
$$ SO(2) \hookrightarrow SO(3) \twoheadrightarrow S^2 $$
(where recall the last map here can be defined as: given a $3 \times 3$ orthogonal matrix, restrict to its first column vector). As remarked (noting, curiously, the offset by ``one position''), the last two terms of the first fibration above are homotopically equivalent to the first two terms of the latter fibration. Thus $\Embeds_{\bdy\textsl{std}}[I,S^2 \times I]$ is homotopically equivalent to the loop space $\Omega(S^2)$ (see e.g. \cite[p. 409]{hatcher}), via the natural map given by projection of $S^2 \times I$ to $S^2$. Hence $\pi_1(\Embeds_{\bdy\textsl{std}}[I,S^2 \times I]) \isom \pi_1(\Omega(S^2)) \isom \pi_2(S^2) \isom \bbZ$. It is fairly easy to see that this isomorphism $\pi_1(\Embeds_{\bdy\textsl{std}} [I, S^2 \times I])) \isom \bbZ$ is given by the degree of the map $(I \times I, \bdy(I \times I)) \rightarrow (S^2 \times I, \point \times I) \rightarrow (S^2, \point)$ that is induced by a $\pi_1$-representative map $I \rightarrow \Embeds_{\bdy\textsl{std}}[I,S^2 \times I]$, with $S^2 \times I \rightarrow S^2$ being the projection.

This completes the proof of Gabai's delightful 4D-LBT.

\bibliographystyle{abbrv}

\vspace{5pt}
email: rde@math.ucla.edu \qquad Errata, comments, suggestions and-or questions are welcomed. All will be read, but no response is promised. This was fun. Thanks, Dave.

\end{document}